\documentclass[letterpaper,11pt,twoside,keywordsasfootnote,addressatend,noinfoline]{article}
\usepackage{fullpage}
\usepackage[english]{babel}
\usepackage{amssymb}
\usepackage{amsmath}
\usepackage{theorem}
\usepackage{epsfig}
\usepackage{mathrsfs}
\usepackage{imsart}
\usepackage{color}
\usepackage{hyperref}

\theorembodyfont{\slshape}
\newtheorem{theorem}{Theorem}

\newtheorem{lemma}[theorem]{Lemma}

\newenvironment{proof}{\noindent{\scshape Proof.}}{\hspace*{2mm} $\square$}

\newcommand{\C}{\mathscr{C}}

\newcommand{\T}{\mathbb{T}}
\newcommand{\n}{\hspace*{-6pt}}

\DeclareMathOperator{\card}{card \,}
\DeclareMathOperator{\bernoulli}{Bernoulli \,}
\DeclareMathOperator{\var}{Var}
\DeclareMathOperator{\diam}{diam}

\begin{document}

\begin{frontmatter}
\title     {Exact insurance premiums for cyber risk \\ of small and medium-sized enterprises}
\runtitle  {Exact insurance premiums for cyber risk of small and medium-sized enterprises}
\author    {Stefano Chiaradonna and Nicolas Lanchier\thanks{This work is partially supported by NSF grant CNS-2000792.}}
\runauthor {Stefano Chiaradonna and Nicolas Lanchier}
\address   {School of Mathematical and Statistical Sciences \\ Arizona State University \\ Tempe, AZ 85287, USA. \\ schiarad@asu.edu \\ nicolas.lanchier@asu.edu}

\begin{abstract} \
 As cyber attacks have become more frequent, cyber insurance premiums have increased, resulting in the need for better modeling of cyber risk.
 Toward this direction, Jevti\'{c} and Lanchier (2020) proposed a dynamic structural model of aggregate loss distribution for cyber risk of small and medium-sized enterprises under the assumption of a tree-based local-area-network topology that consists of the combination of a Poisson process, homogeneous random trees, bond percolation processes, and cost topology.
 Their model assumes that the contagion spreads through the edges of the network with the same fixed probability in both directions, thus overlooking a dynamic cyber security environment implemented in most networks, and their results give an exact expression for the mean of the aggregate loss but only a rough upper bound for the variance.
 In this paper, we consider a bidirectional version of their percolation model in which the contagion spreads through the edges of the network with a certain probability moving toward the lower level assets of the network but with another probability moving toward the higher level assets of the network, which results in a more realistic cyber security environment.
 In addition, our mathematical approach is quite different and leads to exact expressions for both the mean and the variance of the agregate loss, and therefore an exact expression for the insurance premiums.
%  Information is a key component in determining the price of an asset in financial markets, and the main objective of this paper is to study the spread of infection in this context.
%  The network of interactions in financial markets is modeled using a Galton-Watson tree where vertices represent the traders and where two traders are connected by an edge if one of the two traders sells the asset to the other trader.
%  The infection starts from a given vertex and spreads through the edges of the graph going independently from seller to buyer with probability~$p$ and from buyer to seller with probability~$q$.
%  In particular, the set of traders who are aware of the infection is a (bidirectional) bond percolation cluster on the Galton-Watson tree.
%  Using some conditioning techniques and a partition of the cluster of open edges into subtrees, we compute explicitly the first and second moments of the cluster size, i.e., the random number of traders who learn about the infection.
%  We also prove exponential decay of the diameter of the cluster in the subcritical phase.
\end{abstract}



\end{frontmatter}

\maketitle

%%%%%%%%%%%%%%%%%%%%%%%%%%%%%%%%%%%%%%%%%%%%%%%%%%%%%%%%%%%%%%%%%%%%%%%%%%%%%%%%%%%%%%%%%%%%%%%%%%%%%%%%%%%%%%%%%%%%%%%%%%%%%%%%%%%%%%%%%%%%%%%%%%%%%%%%%%%%%%%%%%%%%%%%%%%%

\section{Introduction}
 According to The Institute of Risk Management, cyber risk is ``any risk of financial loss, disruption or damage  to the reputation of an organization from some sort of failure of its information technology systems''~\cite{The_Institute_of_Risk_Management_2018}.
 Following~\cite{jevtic_lanchier_2020}, we define cyber risk due to a data breach as ``the risk of a financial loss caused by a breach of an institution's IT infrastructure by unauthorized parties, resulting in exploitation, taking possession of, or disclosure of data assets''.
 Cyber risk has attracted considerable attention within the past years due to the rapidly growing number of cyber attacks~\cite{1257}.
 In the first six months of 2021, there were 2.5 billion malware attacks and 2.5 trillion intrusion attempts in which an intruder gains or attempts to gain unauthorized access to a system or its network~\cite{1252, intrusion_nist}.
 Malicious actors could steal information by exploiting vulnerable privileges, which are actions a user is permitted to perform on an asset, and accounts, which are a set of privileges given to a user~\cite{1246}.
 Some users, especially domain administrators, have network accounts with privileges that give them greater access to information technology~(IT) resources, such as other systems in the network or the entire network itself because they can bypass critical security settings~\cite{1247, 1248}.
 Once malicious actors have obtained the credentials of an authorized user, they could use the stolen account credentials and privileges to try to elevate their own privileges to a higher administrator-level access and travel from system to system in the network~\cite{1247}.
 Because privileged accounts have greater access to the network, credentials remain one of the most sought-after data assets by attackers~\cite{1247, 1252, 1246}.
 According to the Verizon Data Breach Investigations (2018) report, credentials, phishing, and privileges abuse are three of the top five malicious actions in breaches in~2018~\cite{1260}.
 Out of the~3,841 incidents analyzed in the Verizon Data Breach Investigations (2021) report, of which~1,767 had confirmed data disclosure, the data compromised was credentials at~85\%~\cite{1253}. \\
\indent Due to the increase in cyber risks, there has been more demand for cyber insurance~\cite{1258} as companies have been underestimating the financial impact of cyber risk, so they purchase more coverage or higher limits~\cite{1254}.
 According to~\cite{1258}, analysis of data from S\&P Market Intelligence and National Association of Insurance Commissioners, there was a~60\% increase of cyber insurance policies from~2016 to~2019 and approximately a~50\% increase in the amount of total direct written premiums from \$2.1 billion to \$3.1 billion.
 In the~UK, the number of cyber insurance claims doubled between 2019 and 2020~\cite{1254}.
 For all companies with more than~\$500 million in annual revenue in 2020, the average cyber insurance limits rose by~2\%~\cite{1254}.
 In 2020, there was an increase of about~15\% on average of cyber insurance pricing~\cite{1254}.
 More than half of the brokers surveyed by~\cite{1258} reported that their clients saw a~10-30\% price increase in their cyber insurance premiums.
 More companies are purchasing cyber insurance largely due to the growing number of cyber attacks~\cite{1254} despite the seeing higher insurance prices due to increased severity and frequency of the attacks~\cite{1258}.
 Many small to midsize companies are seeking coverage against data loss, revenue loss due to data breach, legal expenses, and other costs~\cite{herath_herath, betterley, 1257}.
 For small and midsize enterprises, the average cost of a data breach was~\$178 thousand~\cite{1255}, so cyber insurance provides a supposed second line of defense to control and mitigate cyber attacks~\cite{1256}.
 However, pricing cyber risk insurance products is still very new due to its unique characteristics, such as the lack of standard scoring system or actuarial tables for rate making~\cite{1257}.
 Cyber risk, like many other operational risks, does not have readily available experience data because it relies on the organization's network~\cite{segal}.
 In~2015,~\cite{eling_wirfs} used a traditional actuarial approach for calculating the frequency and severity of loss due to a cyber attack;
 however they assumed that the ``causal structure of risks remains relatively stable over time''~\cite{1256, jorion}.
 Due to the inherent nature of technological advancements, cyber risk is dynamic, so reliance on historical cyber risk data could be misleading~\cite{eling_wirfs,1256}.
 Therefore, traditional actuarial modelling is insufficient for cyber insurance.
 Instead, cyber risk should be modelled based on a network structure due to the organization's cyber security environment, which is the resilience of the organization's network~\cite{1259, 1256}.
 Network resilience describes the network's ability to function in the presence of adverse conditions~\cite{1250}.
 Considering the organization's network resilience, one can have a better understanding of an optimal risk management strategy~\cite{1256}. \\
\indent Other papers have taken the approach of calculating the financial loss in cyber risk. Amin (2019) developed a structure for a Bayesian network to model the financial loss as function of the key drivers of risk and resilience ~\cite{1256}. Xu and Hua (2019) studied the cybersecurity risks via epidemic models involving loss functions and pricing strategies as well as providing a review of other mathematical models that calculate cyber risk ~\cite{1257}. Antonio and Indratno (2021) also used an epidemic model on regular networks to simulate the process of a virus spreading and calculating the total loss ~\cite{1259}.  However, none of these models fully utilize the aspect of cyber resilience, which is very much connected with bond percolation ~\cite{jevtic_lanchier_2020, 1250}. \\
% \indent Other papers have taken the approach of calculating the financial loss in cyber risk.
%  \cite{1256} developed a structure for a Bayesian network to model the financial loss as a function of the key drivers of risk and resilience.
%  \cite{1257} studied the cybersecurity risks via epidemic models involving loss functions and pricing strategies.
%  \cite{1259} also used an epidemic model on regular networks to simulate the process of a virus spreading and calculating the total loss.
%  \cite{1257} provides a review of other mathematical models that calculate cyber risk.
%  However, none of these models fully utilize the aspect of cyber resilience, which is very much connected to bond percolation~\cite{jevtic_lanchier_2020, 1250}. \\
\indent The dynamical model designed in~\cite{jevtic_lanchier_2020} consists of several components, including a bond percolation process, to calculate the aggregate loss resulting from consecutive cyber attacks on a tree based local area network of small and midsize enterprises.
 The percolation process assumes that the contagion spreads through the edges of the network with a fixed probability, thus modeling a static cyber security environment.
 In this paper, we extend their model by assuming that the contagion spreads through the edges with a certain probability moving toward the lower levels of the network but with another~(typically smaller) probability moving toward the upper levels of the network, thus resulting in a bidirectional version of the percolation component in~\cite{jevtic_lanchier_2020} that accounts for administrative privileges.
 Administrative privileges make it more difficult for a malicious actor to spread malware throughout the network and to elevate their own privileges~\cite{1248, 1247}.
 By controlling the accounts that allow access to certain accounts, an enterprise improves its ability to protect its more valuable assets~\cite{1246}.
%  One could argue that having multiple layers of internal network firewalls would be sufficient protection, but if a malicious actor gains administrator-level privileges, the malicious actor can disable or circumvent the firewalls~\cite{1249}. Note: Intersting but this does not support our model, on the contrary.
 Therefore, our bidirectional version leads to a more realistic cyber security environment by incorporating administrative privileges, but one could also consider other forms of internal network cyber security such as firewalls \cite{1249}.
 In addition, while~\cite{jevtic_lanchier_2020} only obtained a rough upper bound for the variance of the aggregate loss, our analysis relies on other techniques that lead to an exact expression of the variance, and therefore an exact expression for the insurance premiums.

\section{Model description}
 As previously explained, the model we consider in this paper is a bidirectional version of the model introduced in~\cite{jevtic_lanchier_2020}.
 In this section, we give a rigorous description including five components.
 To begin with, we assume that the cyber attacks occur in continuous time at a constant rate:
\begin{itemize}
 \item[1.] we let~$(N_t)$ be a Poisson process with intensity~$\lambda$,
\end{itemize}
 and assume that the~$i$th attack occurs at time
 $$ T_i = \inf \{t : N_t = i \} \quad \hbox{for} \quad i = 1, 2, \ldots $$
 At the times of the attacks, the local area network consists of a random tree.
 Depending on whether the network is static or dynamic, the network can be fixed once and for all or can consist of a sequence of independent realizations of the random tree, but we point out that our results are not sensitive to this distinction.
 More precisely, in the dynamical context,
\begin{itemize}
 \item[2.] we let~$\T_i = (V_i, E_i)$ be independent realizations of the Galton-Watson tree with radius~$R$, i.e., there are~$k$ vertices connected to the root with probability~$p_k$, then~$k$ additional vertices connected to each of those vertices with probability~$p_k$, and so on, up to generation~$R$.
 To ensure that the tree has radius~$R$ and avoid trivialities, we assume~$p_0 = 0$.
\end{itemize}
 In the terminology of branching processes, the probabilities~$p_k$ are referred to as the offspring distribution, and two vertices connected by an edge are called the parent and the offspring, with the parent being the vertex closer to the root.
 Next, to fix the source of the attack,
\begin{itemize}
 \item[3.] we let~$X_i \in V_i$ be a vertex chosen at random.
\end{itemize}
 The actual distribution of~$X_i$ is unimportant for this work since our objective is to compute the insurance premium as a function of~$X_i$. 
 Now, to model the contagion itself~(how the infection spreads through the network from the source), we use bidirectional bond percolation:
\begin{itemize}
 \item[4.] we let~$p, q \in (0, 1)$, and assume that each edge of the tree is independently open
 $$ \begin{array}{rcl}
    \hbox{with probability~$p$} & \hbox{in the direction parent $\to$ offspring} \vspace*{4pt} \\
    \hbox{with probability~$q$} & \hbox{in the direction offspring $\to$ parent}. \end{array} $$
\end{itemize}
 In other words, each edge is identified to two arrows.
 The arrow going away from the root is open with probability~$p$ whereas the arrow going toward the root is open with probability~$q$.
 As previously mentioned, the distinction between~$p$ and~$q$ is motivated by the presence of administrative privileges suggesting that the infection spreads more easily moving away from the root than toward the root, meaning that~$p > q$.
 The model in~\cite{jevtic_lanchier_2020} corresponds to the particular case~$p = q$.
 Then, the set of infected vertices is the open percolation cluster starting from the source:
 $$ \C_i = \{y \in V_i : \hbox{there is a directed open path~$X_i \to y$} \}. $$
 To define the aggregate loss, the last step is to assign a cost to the percolation cluster:
\begin{itemize}
 \item[5.] we let~$c_{i, y}$ for all~$i > 0$ and~$y \in V_i$ be independent and identically distributed
\end{itemize}
 and think of this random variable as the cost of vertex~$y$.
 Then, the total loss resulting from the~$i$th cyber attack and the aggregate loss up to time~$t$ are given respectively by
 $$ C_i = \sum_{y \in \C_i}  c_{i, y} \quad \hbox{and} \quad L_t = \sum_{i = 1}^{N_t} \,C_i = \sum_{i = 1}^{N_t} \ \sum_{y \in \C_i}  c_{i, y}. $$
 In other words, the loss resulting from the~$i$th cyber attack is the total cost of all the vertices that have been infected during the attack, and the aggregate loss is the cumulative loss resulting from all the cyber attacks that occurred by time~$t$.

%%%%%%%%%%%%%%%%%%%%%%%%%%%%%%%%%%%%%%%%%%%%%%%%%%%%%%%%%%%%%%%%%%%%%%%%%%%%%%%%%%%%%%%%%%%%%%%%%%%%%%%%%%%%%%%%%%%%%%%%%%%%%%%%%%%%%%%%%%%%%%%%%%%%%%%%%%%%%%%%%%%%%%%%%%%%

\section{Main results}
 Our main objective is to compute the mean and variance of~$L_t$ as insurance premiums are calculated from these two quantities.
 To state our results and express the mean and variance of the aggregate loss, we need some key quantities.
 First, we let~$\mu$ and~$\sigma^2$ be respectively the mean and variance of the offspring distribution~(the random number of edges starting from each vertex):
 $$ \mu = \sum_{k = 1}^{\infty} \,k p_k \quad \hbox{and} \quad \sigma^2 = \sum_{k = 1}^{\infty} \,(k - \mu)^2 p_k. $$
 Recalling that the local costs~$c_{i, y}$ are identically distributed, to simplify the notation, we let~$c$ denote their common distribution.
 Similarly, because the consecutive Galton-Watson trees, percolation processes, sources of infection, and local costs are identically distributed, the numbers of infected vertices~$S_i = \card (\C_i)$ are also identically distributed, and we let~$S$ denote the common distribution of the size of the consecutive percolation clusters.
 The model parameters (offspring distribution, distribution of the source, percolation parameters~$p$ and~$q$, and distribution of the local costs) vary from one company to another, and the goal of this paper is not to estimate these parameters.
 Instead, our main objective is to compute explicitly the mean and variance of the aggregate loss, and therefore the insurance premiums, as a function of these parameters. \vspace*{5pt} \\
\noindent {\bf Aggregate loss.}
 The aggregate loss can be expressed using the loss resulting from a single attack by conditioning on the number of attacks.
 Similarly, the loss resulting from a single attack can be expressed using the local cost~$c$ by conditioning on the cluster size~$S$.
 Using also that the number of attacks is Poisson distributed, we obtain the following result.
\begin{theorem}-- \
\label{th:loss}
 The mean and variance of the aggregate loss are given by
 $$ E (L_t) = \lambda t E (S) E (c) \quad \hbox{and} \quad \var (L_t) = \lambda t E (S) \var (c) + \lambda t E (S^2) (E (c))^2. $$
\end{theorem}
 In particular, the theorem shows that the mean and variance of the aggregate loss depends on the first and second moments of the size of the infected cluster therefore, in order to compute the insurance premiums, the next step is to compute those two quantities. \vspace*{5pt} \\
\noindent {\bf First moment.}
 To begin with, we look at the first moment of the cluster size.
 This quantity has been computed explicitly in~\cite{jevtic_lanchier_2020} in the symmetric case~$p = q$ using combinatorial techniques.
 To extend their result to the more general asymmetric case, we partition the set of infected vertices into subtrees and use linearity of the expectation, which gives the following theorem.
\begin{theorem}-- \
\label{th:first}
 The conditional first moment on the tree with radius~$R$ given that the infection starts at distance~$r$ from the root is equal to
 $$ E_r (S) = \frac{1}{1 - \mu p} \bigg(1 + q \bigg(\frac{1 - q^r}{1 - q} \bigg) (1 - p) - (\mu p)^{R - r + 1} \bigg(\frac{1 - pq (1 + (\mu - 1)(\mu pq)^r)}{1 - \mu pq} \bigg) \bigg). $$
\end{theorem}
 Note that setting~$p = q$ in the theorem gives
 $$ E_r (S) = \frac{1}{1 - \mu p} \bigg(1 + p (1 - p^r) - (\mu p)^{R - r + 1} \bigg(\frac{1 - p^2 (1 + (\mu - 1)(\mu p^2)^r)}{1 - \mu p^2} \bigg) \bigg), $$
 which is exactly the expression found in~\cite[Theorem~4]{jevtic_lanchier_2020}, but we point out that, even though our result extends the result in~\cite{jevtic_lanchier_2020} to the asymmetric case, our approach leads to a much shorter and more elegant proof.
 More precisely, the proof in~\cite{jevtic_lanchier_2020} relies on a tedious combinatorial argument that consists in counting the number of open paths of a given length starting from the source of the infection whereas our proof consists in finding the infected vertices along the path going from the source of the infection to the root of the tree and then partitioning the cluster of infected vertices into (disjoint) subtrees starting from each of these vertices. \vspace*{5pt} \\
\noindent {\bf Second moment.}
 Our approach to compute the second moment is similar but relies in addition on independence.
 More precisely, writing again the cluster of infected vertices as a disjoint union of subtrees, the second moment can be computed using that the sizes of these subtrees are independent random variables.
 To express the second moment of the cluster size, we let
\begin{equation}
\label{eq:parameters-1}
  \mu_+ = \mu p \quad \hbox{and} \quad \mu_- = (\mu - 1) p,
\end{equation}
 quantities that will be interpreted later as the mean number of infected offspring in certain subtrees of the local area network, and
\begin{equation}
\label{eq:parameters-2}
  \sigma_+^2 = p (1 - p) \mu + p^2 \sigma^2 \quad \hbox{and} \quad \sigma_-^2 = p (1 - p)(\mu - 1) + p^2 \sigma^2,
\end{equation}
 quantities that will be interpreted later as the variance of the number of infected offspring in certain subtrees of the local area network.
 For all~$j = 0, 1, \ldots, r$, we also define
\begin{equation}
\label{eq:parameters-3}
  \begin{array}{rcl}
  \mu_{1, j} & \n = \n & \displaystyle \frac{1 - \mu_+^{R - j + 1}}{1 - \mu_+} \vspace*{4pt} \\
  \mu_{2, j} & \n = \n & \displaystyle \frac{\sigma_+^2}{(1 - \mu_+)^2} \ \bigg(\frac{1 - \mu_+^{2 (R - j) + 1}}{1 - \mu_+} - (2 (R - j) + 1) \mu_+^{R - j} \bigg) + \bigg(\frac{1 - \mu_+^{R - j + 1}}{1 - \mu_+} \bigg)^2. \end{array}
\end{equation}
\begin{theorem}-- \
\label{th:second}
 The conditional second moment on the tree with radius~$R$ given that the infection starts at distance~$r$ from the root is equal to
 $$ \begin{array}{rcl}
    \displaystyle E_r (S^2) = \sum_{k = 0}^r \bigg(\mu_{2, r} & \n + \n &
    \displaystyle 2 \mu_{1, r} \ \sum_{i = 1}^k \ (1 + \mu_- \mu_{1, r - i + 1}) \vspace*{-4pt} \\ & \n + \n &
    \displaystyle \sum_{i = 1}^k \ (1 + 2 \mu_- \mu_{1, r - i + 1} + \mu_- \mu_{2, r - i + 1} + (\sigma_-^2 + \mu_-^2 - \mu_-)(\mu_{1, r - i + 1})^2) \vspace*{0pt} \\ & \n + \n &
    \displaystyle \sum_{i \neq j} \,(1 + \mu_- \mu_{1, r - i + 1})(1 + \mu_- \mu_{1, r - j + 1}) \bigg) \,q_k \end{array} $$
 where~$q_k = q^k (1 - q)$ for~$k = 0, 1, \ldots, r - 1$, and~$q_r = q^r$.
\end{theorem}
 Theorem~\ref{th:second} is the main contribution of this work as it gives an exact expression of the second moment, which leads to an exact pricing of the standard deviation principle, whereas~\cite[Theorem~5]{jevtic_lanchier_2020} only derived a rough upper bound for this pricing.
 Although the expression is implicit, it can be computed explicitly using a computer program~(see Figure~\ref{fig:second}) for each set of parameters. \\
\indent We also point out that the expressions in both Theorems~\ref{th:first} and~\ref{th:second} can be simplified when the local area network consists on the infinite Galton-Watson tree.
 In this case, the percolation process is supercritical when~$\mu p > 1$, meaning that the cluster of infected vertices is infinite with positive probability, so the first and second moments are both infinite.
 In the subcritical phase~$\mu p < 1$, it follows from the monotone convergence theorem that the first and second moments of the cluster size on the infinite tree can be obtained by taking the limit as~$R \to \infty$ in both theorems.
 In particular, in the infinite case, the first moment reduces to
 $$ E_r (S) = \frac{1}{1 - \mu p} \bigg(1 + q \bigg(\frac{1 - q^r}{1 - q} \bigg) (1 - p) \bigg) $$
 while using some algebra we get
\begin{equation}
\label{eq:moment-2}
  \begin{array}{rcl}
    \displaystyle E_r (S^2) & \n = \n &
    \displaystyle \frac{1}{(1 - \mu p)^2} \ \bigg(1 + \frac{p (1 - p) \mu + p^2 \sigma^2}{1 - \mu p} \bigg) \vspace*{8pt} \\ && \hspace*{5pt} + \
    \displaystyle \bigg(1 + \frac{2 (1 + (\mu - 1) p)}{1 - \mu p} + \frac{2 (\mu - 1) p + p (1 - p)(\mu - 1) + p^2 \sigma^2 + (\mu - 1)^2 p^2}{(1 - \mu p)^2} \vspace*{8pt} \\ && \hspace*{100pt} + \
    \displaystyle \frac{(p (1 - p) \mu + p^2 \sigma^2)(\mu - 1) p}{(1 - \mu p)^3} \bigg) \ q \bigg(\frac{1 - q^r}{1 - q} \bigg) \vspace*{8pt} \\ && \hspace*{5pt} + \
    \displaystyle \bigg(1 + \frac{(\mu - 1) p}{1 - \mu p} \bigg)^2 \ \frac{2q^2 (1 - rq^{r - 1} + (r - 1) q^r)}{(1 - q)^2} \end{array}
\end{equation}
 for the second moment on the infinite tree (see Section~\ref{sec:moment-2} for a proof). \vspace*{5pt} \\
\noindent {\bf Exponential decay of the diameter.}
 Another quantity of interest which also accounts for the geometry of the set of infected vertices is the diameter of the cluster~$\C$ defined as the maximum graph distance~(number of edges) between any two infected vertices:
 $$ \diam (\C) = \max \,\{d (x, y) : x, y \in \C \}. $$
 In this case, studying the spread of the infection from a dynamical point of view starting from the highest infected vertex and moving one generation down the tree at each time step, we can prove an exponential decay.
 More precisely, we have the following theorem.
\begin{theorem}-- \
\label{th:decay}
 Let~$\mu p < 1$.
 Then, the conditional probability that the diameter is larger than~$2n$ given that the infection starts at distance~$r$ from the root is
 $$ P_r (\diam (\C) \geq 2n) \leq \frac{1 - (q / \mu p)^{r + 1}}{1 - (q / \mu p)} \ (\mu p)^n \quad \hbox{for all} \quad n > r. $$
\end{theorem}
 The theorem indeed implies that, in the subcritical phase~$\mu p < 1$, the tail distribution of the diameter of the cluster of infected vertices decays exponentially. \vspace*{5pt} \\
\noindent {\bf Insurance premiums.}
\begin{figure}[t!]
\centering
\scalebox{0.30}{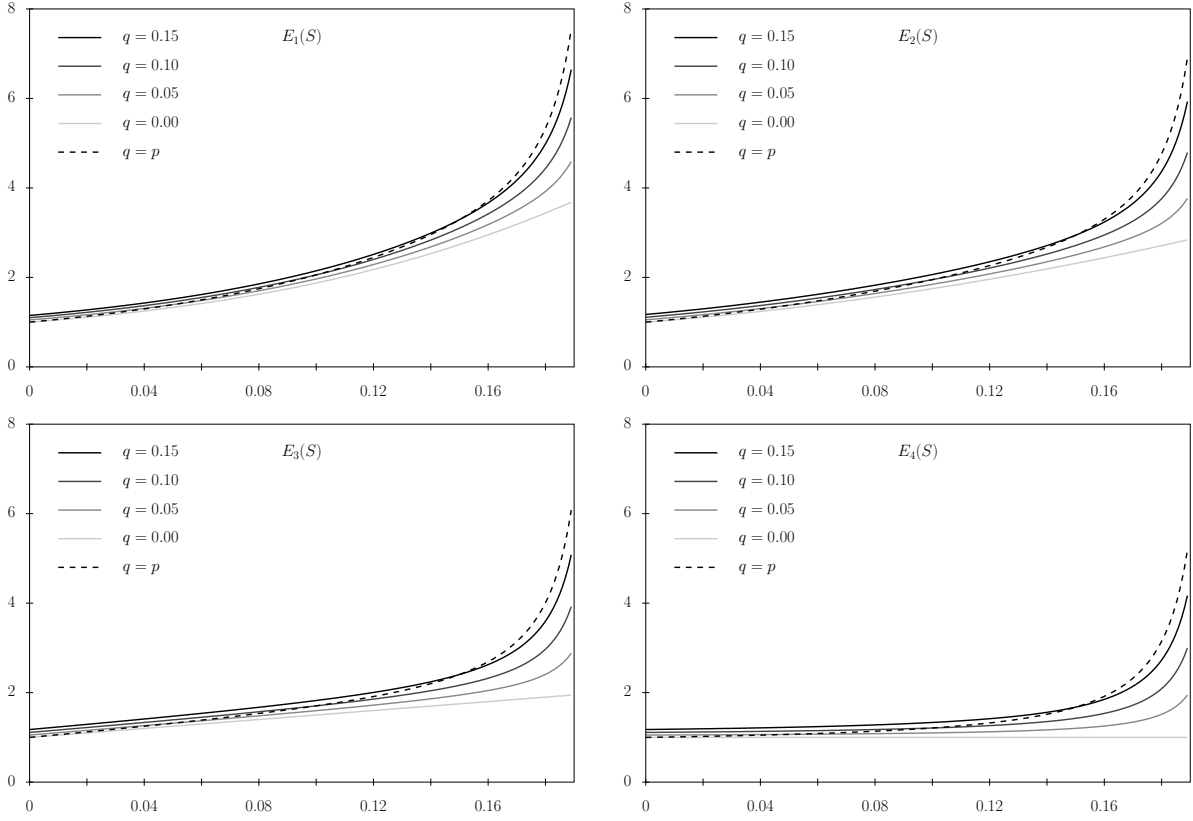}
\caption{\upshape{First moment of the cluster size as a function of~$p$ for various values of~$q$ and~$r$.
 In each picture, the mean and variance of the offspring distribution are~$\mu = \sigma^2 = 5$ and the radius of the tree is~$R = 4$.}}
\label{fig:first}
\end{figure}
\begin{figure}[t!]
\centering
\scalebox{0.30}{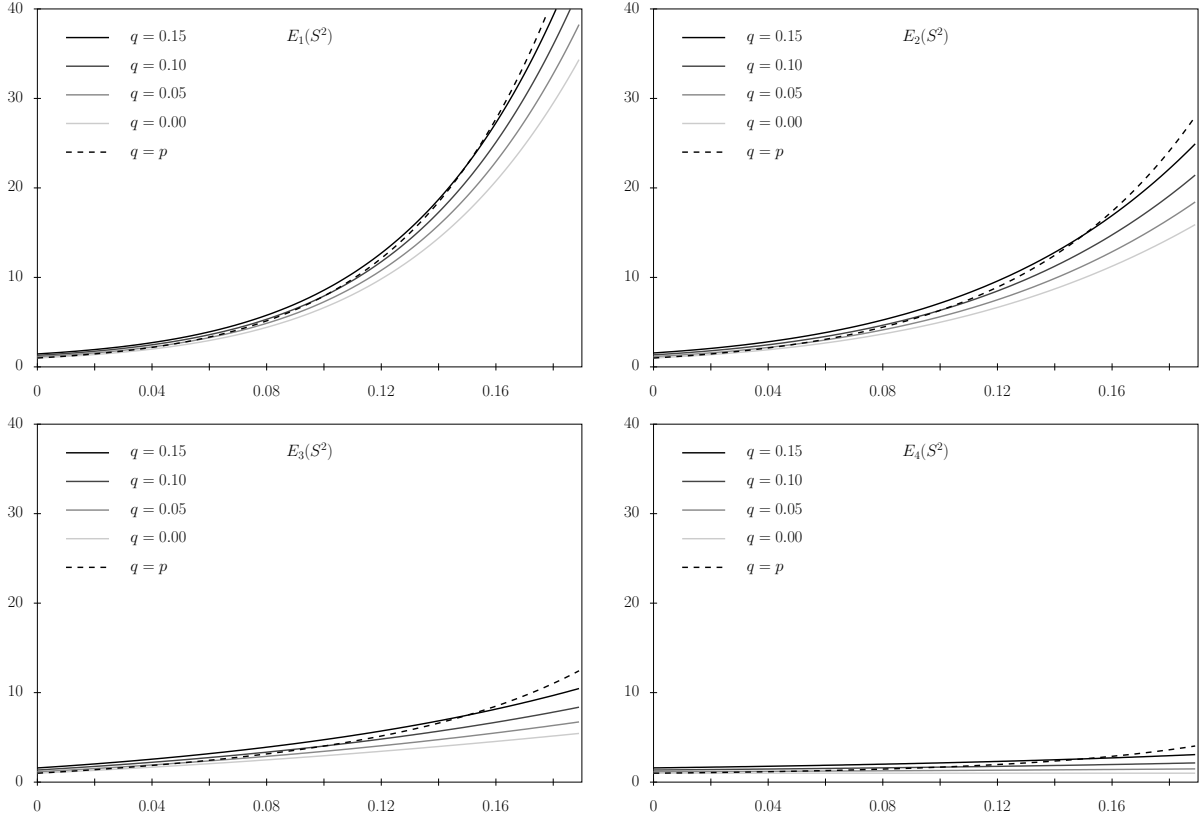}
\caption{\upshape{Second moment of the cluster size as a function of~$p$ for various values of~$q$ and~$r$.
 In each picture, the mean and variance of the offspring distribution are~$\mu = \sigma^2 = 5$ and the radius of the tree is~$R = 4$.}}
\label{fig:second}
\end{figure}
 Combining Theorems~\ref{th:loss}--\ref{th:second}, we can now perform an exact pricing of cyber risk insurance.
 Following~\cite{jevtic_lanchier_2020}, we consider the three pricing principles
 $$ \begin{array}{rl}
    \hbox{actuarial fair premium:} & P = E (L_1) \vspace*{4pt} \\
    \hbox{expectation principle:} & P = E (L_1) + \delta E (L_1) \vspace*{4pt} \\
    \hbox{standard deviation principle:} & P = E (L_1) + \delta \sqrt{\var (L_1)} \end{array} $$
 where~$L_1$ represents the aggregate loss per unit of time.
 According to Theorem~\ref{th:loss},
\begin{equation}
\label{eq:unit-loss}
  E (L_1) = \lambda E (S) E (c) \quad \hbox{and} \quad \sqrt{\var (L_1)} = \sqrt{\lambda E (S) \var (c) + \lambda E (S^2) (E (c))^2},
\end{equation}
 showing that all three insurance premiums are functions of the system parameters and the first and second moments of the cluster size.
 In particular, for each set of parameters (the rate of occurrence of the cyber attacks, the mean and variance of the offspring distribution and the local costs, etc.), all three premiums can be computed exactly using Theorems~\ref{th:first} and~\ref{th:second}.
 Because the expressions of the first two moments in the theorems are quite complicated, we refer the reader to Figures~\ref{fig:first} and~\ref{fig:second} for pictures of the first and second moments as functions of the parameters~$p, q$ and~$r$, for a fixed value of the mean and variance of the offspring distribution.
 The dashed curves in the pictures correspond to the symmetric case~$p = q$ considered in~\cite{jevtic_lanchier_2020} where the first moment was computed exactly while only a very rough upper bound for the second moment was derived.
 In contrast, the approach used in this paper leads to exact insurance premiums for all three principles. \\
\indent Not surprisingly, it follows from~\eqref{eq:unit-loss} that all three premiums are nondecreasing with respect to the rate~$\lambda$ and the mean and variance of the local costs.
 The figures also show that both the first and the second moments of the cluster size are nondecreasing with respect to~$p$ and~$q$ and so are the three premiums.
 This result is intuitively clear and can be proved rigorously using a popular technique in probability theory called coupling, which consists in this case in constructing bond percolation processes with different parameters on the same probability space.
 We note however that the cluster size and the insurance premiums are not always increasing or always decreasing with respect to~$r$, the distance from the root to the source of the infection.
 Indeed,
\begin{itemize}
 \item When~$p = 1$ and~$q = 0$, the set of infected vertices consists of the subtree starting from the source of the infection going away from the root are therefore, in this case, the cluster size and the premiums are decreasing with respect to~$r$. \vspace*{4pt}
 \item When~$p = 0$ and~$q = 1$, the set of infected vertices consists of the unique path going from the source of the infection to the root of the tree are therefore, in this case, the cluster size and the premiums are increasing with respect to~$r$.
\end{itemize}
 The rest of this paper is devoted to the proof of the four theorems.

%%%%%%%%%%%%%%%%%%%%%%%%%%%%%%%%%%%%%%%%%%%%%%%%%%%%%%%%%%%%%%%%%%%%%%%%%%%%%%%%%%%%%%%%%%%%%%%%%%%%%%%%%%%%%%%%%%%%%%%%%%%%%%%%%%%%%%%%%%%%%%%%%%%%%%%%%%%%%%%%%%%%%%%%%%%%

\section{Proof of Theorem~\ref{th:loss} (aggregate loss)}
 To prove Theorem~\ref{th:loss}, we first observe that, because the Galton-Watson trees, percolation processes and local costs are independent and identically distributed across time, the consecutive costs~$C_i$ are independent and identically distributed as well.
 In particular, letting~$C$ be the common distribution of the random variables~$C_i$ and conditioning on the number of cyber attacks~$N_t$, we get
\begin{equation}
\label{eq:loss-1}
  \begin{array}{rcl}
       E (L_t \,| \,N_t = n) & \n = \n & E (C_1 + \cdots + C_{N_t} \,| \,N_t = n) = n E (C) \vspace*{4pt} \\
    \var (L_t \,| \,N_t = n) & \n = \n & \var (C_1 + \cdots + C_{N_t} \,| \,N_t = n) = n \var (C). \end{array}
\end{equation}
 The first equation in~\eqref{eq:loss-1} implies that
\begin{equation}
\label{eq:loss-2}
  E (L_t) = E (E (L_t \,| \,N_t)) = E (N_t E (C)) = E (N_t) E (C)
\end{equation}
 while using also the second equation in~\eqref{eq:loss-1} and the law of total variance,
\begin{equation}
\label{eq:loss-3}
  \begin{array}{rcl}
  \var (L_t) & \n = \n & E (\var (L_t \,| \,N_t)) + \var (E (L_t \,| \,N_t)) = E (N_t \var (C)) + \var (N_t E (C)) \vspace*{4pt} \\
             & \n = \n & E (N_t) \var (C) + \var (N_t) (E (C))^2. \end{array}
\end{equation}
 Using that the local costs are independent and identically distributed across the local area network, and conditioning on the size~$S$ of a single cyber attack, we also have
\begin{equation}
\label{eq:loss-4}
  E (C \,| \,S = s) = s E (c) \quad \hbox{and} \quad \var (C \,| \,S = s) = s \var (c).
\end{equation}
 The first equation in~\eqref{eq:loss-4} implies that
\begin{equation}
\label{eq:loss-5}
  E (C) = E (E (C \,| \,S)) = E (S E (c)) = E (S) E (c)
\end{equation}
 while using also the second equation in~\eqref{eq:loss-4} and the law of total variance,
\begin{equation}
\label{eq:loss-6}
  \begin{array}{rcl}
  \var (C) & \n = \n & E (\var (C \,| \,S)) + \var (E (C \,| \,S)) = E (S \var (c)) + \var (S E (c)) \vspace*{4pt} \\
           & \n = \n & E (S) \var (c) + \var (S) (E (c))^2. \end{array}
\end{equation}
 Finally, using that~$E (N_t) = \var (N_t) = \lambda t$, and combining~\eqref{eq:loss-2} and~\eqref{eq:loss-5}, we get
 $$ E (L_t) = \lambda t E (C) = \lambda t E (S) E (c). $$
 Combining~\eqref{eq:loss-3}, \eqref{eq:loss-5} and~\eqref{eq:loss-6}, and using that~$\var (S) + (E (S))^2 = E (S^2)$, we get
 $$ \begin{array}{rcl}
    \var (L_t) & \n = \n & \lambda t \var (C) + \lambda t (E (C))^2 \vspace*{4pt} \\
               & \n = \n & \lambda t E (S) \var (c) + \lambda t \var (S) (E (c))^2 + \lambda t (E (S) E (c))^2 \vspace*{4pt} \\
               & \n = \n & \lambda t E (S) \var (c) + \lambda t E (S^2) (E (c))^2. \end{array} $$
 This completes the proof of the theorem.

%%%%%%%%%%%%%%%%%%%%%%%%%%%%%%%%%%%%%%%%%%%%%%%%%%%%%%%%%%%%%%%%%%%%%%%%%%%%%%%%%%%%%%%%%%%%%%%%%%%%%%%%%%%%%%%%%%%%%%%%%%%%%%%%%%%%%%%%%%%%%%%%%%%%%%%%%%%%%%%%%%%%%%%%%%%%

\section{Partition into disjoint subtrees}
 To get ready for the proofs of Theorems~\ref{th:first} and~\ref{th:second} in the next two sections, we first describe the partition of the set of infected vertices into disjoint subtrees and collect several useful preliminary results that explain the parameters introduced in~\eqref{eq:parameters-1}--\eqref{eq:parameters-3}.
 More precisely, we study the distribution of the random number of subtrees and compute the first and second moment of the size of these subtrees.
 From now on, we assume that the infection starts at a vertex~$x$ with~$d (0, x) = r$.
 By spherical symmetry, the specific choice of~$x$ is unimportant as long as the vertex is at distance~$r$ from the root.
 There is a unique directed path
 $$ x_0 = 0 \to x_1 \to x_2 \to \cdots \to x_{r - 1} \to x_r = x $$
 of length~$r$ going from the root to vertex~$x$ and we let
 $$ D = \max \,\{i = 0, 1, \ldots, r : x_{r - i} \ \hbox{is infected} \}. $$
 This is the distance between the source of the infection and the highest infected vertex, and we refer the reader to Figure~\ref{fig:tree} for a picture.
 The following lemma gives preliminary results about the random variable~$D$ that will be useful later to prove the theorems.
\begin{figure}[t!]
\centering
\scalebox{0.65}{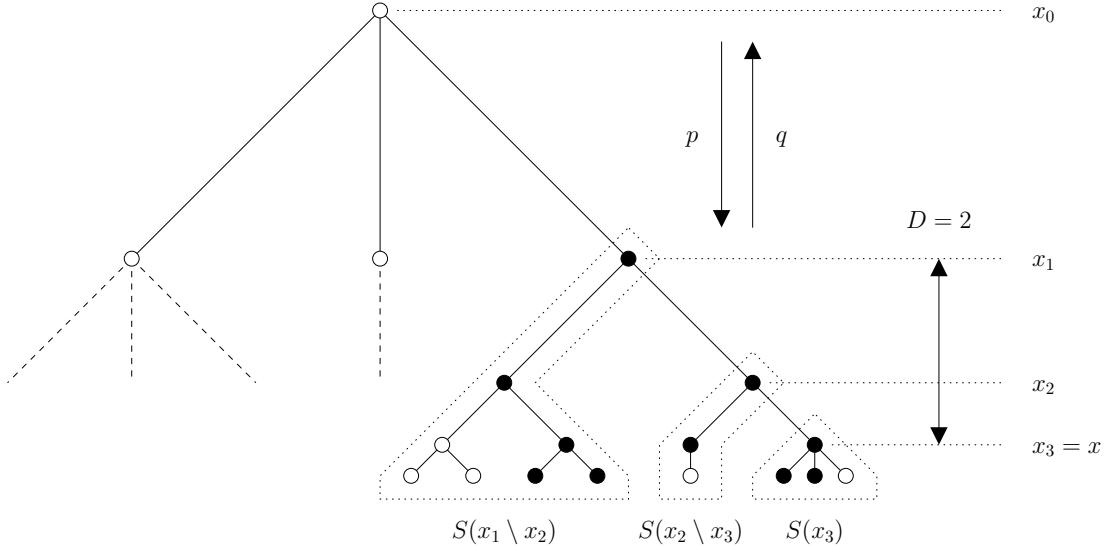}
\caption{\upshape{Picture of the partition into disjoint subtrees used to prove Theorems~\ref{th:first} and~\ref{th:second}.
 The black vertices represent the set of infected vertices while the white vertices are not infected.
 In our example, the infection starts from~$x_3$ and spreads up to~$x_1$, which results in a partition of the cluster of infected vertices into three disjoint subtrees.
 Starting from the source of the infection, the numbers of infected vertices in the subtrees are~3, 2 and 5, respectively.}}
\label{fig:tree}
\end{figure}
\begin{lemma}-- \
\label{lem:D}
 We have
 $$ E (D) = q \bigg(\frac{1 - q^r}{1 - q} \bigg) \quad \hbox{and} \quad E (D (D - 1)) = \frac{2q^2 (1 - rq^{r - 1} + (r - 1) q^r)}{(1 - q)^2}. $$
\end{lemma}
\begin{proof}
 Because the infection spreads toward the root of the random tree with probability~$q$ independently through each of the edges, we have
\begin{equation}
\label{eq:D}
  P (D = k) = q^k (1 - q) \quad \hbox{for} \ k = 0, 1, \ldots, r - 1, \quad \hbox{and} \quad P (D = r) = q^r.
\end{equation}
 In particular, using that
 $$ \sum_{k = 1}^{r - 1} \,k x^{k - 1} =
    \frac{\partial}{\partial x} \bigg(\sum_{k = 0}^{r - 1} x^k \bigg) =
    \frac{\partial}{\partial x} \bigg(\frac{1 - x^r}{1 - x} \bigg) = \frac{1 - rx^{r - 1} + (r - 1) x^r}{(1 - x)^2}, $$
 we deduce that the first moment is given by
 $$ \begin{array}{rcl}
      E (D) & \n = \n &
    \displaystyle \sum_{k = 0}^{r - 1} \,k q^k (1 - q) + r q^r = q (1 - q) \sum_{k = 1}^{r - 1} \,k q^{k - 1} + r q^r \vspace*{4pt} \\ & \n = \n &
    \displaystyle q (1 - q) \bigg(\frac{1 - rq^{r - 1} + (r - 1) q^r}{(1 - q)^2} \bigg) + r q^r = q \bigg(\frac{1 - q^r}{1 - q} \bigg). \end{array} $$
 Similarly, we have
 $$ \begin{array}{rcl}
    \displaystyle \sum_{k = 2}^{r - 1} \,k (k - 1) x^{k - 2} & \n = \n &
    \displaystyle \frac{\partial^2}{\partial x^2} \bigg(\sum_{k = 0}^{r - 1} x^k \bigg) =
    \displaystyle \frac{\partial}{\partial x} \bigg(\frac{1 - rx^{r - 1} + (r - 1) x^r}{(1 - x)^2} \bigg) \vspace*{8pt} \\ & \n = \n &
    \displaystyle \frac{2 - r (r - 1) x^{r - 2} + 2r (r - 2) x^{r - 1} - (r - 1)(r - 2) x^r}{(1 - x)^3} \end{array} $$
 from which it follows that
 $$ \begin{array}{rcl}
    \displaystyle E (D (D - 1)) & \n = \n &
    \displaystyle \sum_{k = 0}^{r - 1} \,k (k - 1) q^k (1 - q) + r (r - 1) q^r \vspace*{4pt} \\ & \n = \n &
    \displaystyle q^2 (1 - q) \ \frac{2 - r (r - 1) q^{r - 2} + 2r (r - 2) q^{r - 1} - (r - 1)(r - 2) q^r}{(1 - q)^3} \vspace*{4pt} \\ && \hspace*{50pt} + \
    \displaystyle q^2 \ \frac{r (r - 1) q^{r - 2} - 2r (r - 1) q^{r - 1} + r (r - 1) q^r}{(1 - q)^2} \vspace*{4pt} \\ & \n = \n &
    \displaystyle \frac{2q^2 (1 - rq^{r - 1} + (r - 1) q^r)}{(1 - q)^2}. \end{array} $$
 This completes the proof.
\end{proof} \\ \\
 Next, for~$j = 0, 1, \ldots, r$, we define the random variables
\begin{equation}
\label{eq:subtrees}
  \begin{array}{rcl}
    S (x_j) & \n = \n & \hbox{number of infected vertices in the subtree starting at~$x_j$} \vspace*{4pt} \\
    S (x_j \setminus x_{j + 1}) & \n = \n & \hbox{number of infected vertices in the subtree starting at~$x_j$} \\ &&
                                            \hbox{but excluding the subtree starting at~$x_{j + 1}$.} \end{array}
\end{equation}
 For instance, in the realization shown in Figure~\ref{fig:tree}, we have
 $$ S (x_3) = 3, \quad S (x_2 \setminus x_3) = 2, \quad S (x_1 \setminus x_2) = 5, \quad S (x_0 \setminus x_1) = 0. $$
 To estimate the first and second moments of the size of the infected cluster later, we now compute the first and second moments of the random variables in~\eqref{eq:subtrees}.
 To do this, let~$\xi_i = \bernoulli (p)$ be independent, and let~$Y$ be the offspring distribution, i.e., the random variable describing the number of edges starting from a given vertex and going away from the root.
 In particular,
 $$ X_+ = \xi_1 + \xi_2 + \cdots + \xi_Y \quad \hbox{and} \quad X_- = \xi_1 + \xi_2 + \cdots + \xi_{Y - 1} $$
 are the random variables describing the number of infected offspring of a given vertex and the number of infected offspring of a given vertex excluding a given offspring, respectively.
 The next two lemmas give the mean and the variance of these two random variables.
\begin{lemma}-- \
\label{lem:mean}
 We have~$E (X_+) = \mu_+$ and~$E (X_-) = \mu_-$ as defined in~\eqref{eq:parameters-1}.
\end{lemma}
\begin{proof}
 Conditioning on~$Y$, we get
 $$ E (X_+) = E (E (X_+ \,| \,Y)) = E (E (\xi_1 + \cdots + \xi_Y \,| \,Y)) = E (Y) E (\xi_n) = \mu p = \mu_+. $$
 Similarly, for the mean of~$X_-$,
 $$ E (X_-) = E (E (X_- \,| \,Y)) = E (E (\xi_1 + \cdots + \xi_{Y - 1} \,| \,Y)) = E (Y - 1) E (\xi_n) = (\mu - 1) p = \mu_-. $$
 This completes the proof.
\end{proof}
\begin{lemma}-- \
\label{lem:variance}
 We have~$\var (X_+) = \sigma_+^2$ and~$\var (X_-) = \sigma_-^2$ as defined in~\eqref{eq:parameters-2}.
\end{lemma}
\begin{proof}
 Conditioning on~$Y$ and using the law of total variance, we get
 $$ \begin{array}{rcl}
    \var (X_+) & \n = \n & E (\var (X_+ \,| \,Y)) + \var (E (X_+ \,| \,Y)) \vspace*{4pt} \\
               & \n = \n & E (p (1 - p) Y) + \var (pY) = p (1 - p) \mu + p^2 \sigma^2 = \sigma_+^2. \end{array} $$
 Similarly, for the variance of~$X_-$,
 $$ \begin{array}{rcl}
    \var (X_-) & \n = \n & E (\var (X_- \,| \,Y)) + \var (E (X_- \,| \,Y)) \vspace*{4pt} \\
               & \n = \n & E (p (1 - p)(Y - 1)) + \var (p (Y - 1)) = p (1 - p)(\mu - 1) + p^2 \sigma^2 = \sigma_-^2. \end{array} $$
 This completes the proof.
\end{proof} \\ \\
 Now, applying~Lemmas~8 and~9 from~\cite{jevtic_lanchier_2020} implies that the first and second moments of the first set of random variables in~\eqref{eq:subtrees} are given respectively by
 $$ \begin{array}{rcl}
       E (S (x_j)) & \n = \n & \displaystyle \frac{1 - \mu_+^{R - j + 1}}{1 - \mu_+} \vspace*{4pt} \\
     E (S (x_j)^2) & \n = \n & \displaystyle \frac{\sigma_+^2}{(1 - \mu_+)^2} \ \bigg(\frac{1 - \mu_+^{2 (R - j) + 1}}{1 - \mu_+} - (2 (R - j) + 1) \mu_+^{R - j} \bigg) + \bigg(\frac{1 - \mu_+^{R - j + 1}}{1 - \mu_+} \bigg)^2, \end{array} $$
 which are the expressions in~\eqref{eq:parameters-3}.
 In particular, we have the following lemma.
\begin{lemma}-- \
\label{lem:plus}
 For all~$j = 0, 1, \ldots, r$, we have~$E (S (x_j)) = \mu_{1, j}$ and~$E (S (x_j)^2) = \mu_{2, j}$.
\end{lemma}
 We now look at the second set of random variables in~\eqref{eq:subtrees}.
\begin{lemma}-- \
\label{lem:minus}
 For all~$j = 0, 1, \ldots, r$, we have~$E (S (x_j \setminus x_{j + 1})) = 1 + \mu_- \mu_{1, j + 1}$ and
 $$ E (S (x_j \setminus x_{j + 1})^2) = 1 + 2 \mu_- \mu_{1, j + 1} + \mu_- \mu_{2, j + 1} + (\sigma_-^2 + \mu_-^2 - \mu_-)(\mu_{1, j + 1})^2. $$
\end{lemma}
\begin{proof}
 Letting~$y_1, y_2, \ldots, y_X$ be the infected offspring of~$x_j$ other than~$x_{j + 1}$,
\begin{equation}
\label{eq:partition}
 S (x_j \setminus x_{j + 1}) = 1 + S (y_1) + \cdots + S (y_X) \quad \hbox{and} \quad X = X_- \ \hbox{in distribution}.
\end{equation}
 In particular, conditioning on~$X$ and using Lemmas~\ref{lem:mean} and~\ref{lem:plus}, we obtain
 $$ \begin{array}{rcl}
      E (S (x_j \setminus x_{j + 1})) & \n = \n &
      E (E (1 + S (y_1) + \cdots + S (y_X) \,| \,X)) \vspace*{4pt} \\ & \n = \n &
      1 + E (X) E (S (y_1)) = 1 + \mu_- E (S (x_{j + 1})) = 1 + \mu_- \mu_{1, j + 1}. \end{array} $$
 Taking the square in~\eqref{eq:partition},
 $$ \begin{array}{rcl}
      S (x_j \setminus x_{j + 1})^2 & \n = \n &
    \displaystyle 1 + 2 \,\sum_{n = 1}^X S (y_n) + \bigg(\sum_{n = 1}^X S (y_n) \bigg)^2 \vspace*{4pt} \\ & \n = \n & 
    \displaystyle 1 + 2 \,\sum_{n = 1}^X S (y_n) + \sum_{n = 1}^X S (y_n)^2 + \sum_{n \neq m} S (y_n) S (y_m), \end{array} $$
 then conditioning on~$X$ and using independence as well as Lemmas~\ref{lem:mean}--\ref{lem:plus},
 $$ \begin{array}{rcl}
     E (S (x_j \setminus x_{j + 1})^2) & \n = \n &
     E (E (S (x_j \setminus x_{j + 1})^2 \,| \,X)) \vspace*{4pt} \\ & \n = \n &
     1 + 2 E (X) E (S (y_n)) + E (X) E (S (y_n)^2) + E (X (X - 1)) E (S (y_n))^2 \vspace*{4pt} \\ & \n = \n &
     1 + 2 \mu_- E (S (x_{j + 1})) + \mu_- E (S (x_{j + 1})^2) + (\sigma_-^2 + \mu_-^2 - \mu_-) E (S (x_{j + 1}))^2 \vspace*{4pt} \\ & \n = \n &
     1 + 2 \mu_- \mu_{1, j + 1} + \mu_- \mu_{2, j + 1} + (\sigma_-^2 + \mu_-^2 - \mu_-)(\mu_{1, j + 1})^2. \end{array} $$
 This completes the proof.
\end{proof}

%%%%%%%%%%%%%%%%%%%%%%%%%%%%%%%%%%%%%%%%%%%%%%%%%%%%%%%%%%%%%%%%%%%%%%%%%%%%%%%%%%%%%%%%%%%%%%%%%%%%%%%%%%%%%%%%%%%%%%%%%%%%%%%%%%%%%%%%%%%%%%%%%%%%%%%%%%%%%%%%%%%%%%%%%%%%

\section{Proof of Theorem \ref{th:first} (first moment)}
 Using the previous lemmas, we are now ready to prove Theorem~\ref{th:first}.
 To begin with, observe that, on the event~$D = k$, the total number of infected vertices can be written as
\begin{equation}
\label{eq:size}
  S = S (x_r) + \sum_{i = r - k}^{r - 1} S (x_i \setminus x_{i + 1}) = S (x_r) + \sum_{i = 1}^k \,S (x_{r - i} \setminus x_{r - i + 1}).
\end{equation}
 Then, conditioning on~$D$ and using Lemma~\ref{lem:minus}, we get
\begin{equation}
\label{eq:thm1a}
  \begin{array}{rcl} E_r (S) & \n = \n &
  \displaystyle \sum_{k = 0}^r \,E (S \,| \,D = k) \,P (D = k) \vspace*{4pt} \\ & \n = \n &
  \displaystyle \sum_{k = 0}^r \,\bigg(E (S (x_r)) + \sum_{i = 1}^k \,E (S (x_{r - i} \setminus x_{r - i + 1})) \bigg) P (D = k) \vspace*{4pt} \\ & \n = \n &
  \displaystyle \sum_{k = 0}^r \,\bigg(E (S (x_r)) + \sum_{i = 0}^{k - 1} \ (1 + \mu_- E (S (x_{r - i})) \bigg) P (D = k) \vspace*{4pt} \\ & \n = \n &
  \displaystyle E (S (x_r)) + E (D) + \mu_- \ \sum_{k = 1}^r \ \sum_{i = 0}^{k - 1} \,E (S (x_{r - i})) \,P (D = k). \end{array}
\end{equation}
 Exchanging the two sums, and using Lemma~\ref{lem:plus}, we obtain
\begin{equation}
\label{eq:thm1b}
  \begin{array}{l}
  \displaystyle \sum_{k = 1}^r \ \sum_{i = 0}^{k - 1} \,E (S (x_{r - i})) \,P (D = k) =
  \displaystyle \sum_{i = 0}^{r - 1} \ \sum_{k = i + 1}^r \,E (S (x_{r - i})) \,P (D = k) \vspace*{4pt} \\ \hspace*{40pt} =
  \displaystyle \sum_{i = 0}^{r - 1} \,E (S (x_{r - i})) \,P (D > i) =
  \displaystyle \sum_{i = 0}^{r - 1} \,q^{i + 1} \,E (S (x_{r - i})) =
  \displaystyle \sum_{i = 0}^{r - 1} \,q^{i + 1} \,\mu_{1, r - i} \vspace*{8pt} \\ \hspace*{40pt} =
  \displaystyle \sum_{i = 0}^{r - 1} \,q^{i + 1} \bigg(\frac{1 - \mu_+^{R - r + i + 1}}{1 - \mu_+} \bigg) =
  \displaystyle \frac{1}{1 - \mu_+} \bigg(q \ \sum_{i = 0}^{r - 1} \,q^i - q \,\mu_+^{R - r + 1} \ \sum_{i = 0}^{r - 1} \,(q \mu_+)^i \bigg) \vspace*{8pt} \\ \hspace*{40pt} =
  \displaystyle \frac{q}{1 - \mu_+} \bigg(\bigg(\frac{1 - q^r}{1 - q} \bigg) - \mu_+^{R - r + 1} \bigg(\frac{1 - (q \mu_+)^r}{1 - q \mu_+} \bigg) \bigg). \end{array}
\end{equation}
 Combining~\eqref{eq:thm1a} and~\eqref{eq:thm1b}, and using Lemmas~\ref{lem:D} and~\ref{lem:plus}, we deduce that
 $$ \begin{array}{rcl} E_r (S) & \n = \n &
    \displaystyle \frac{1 - \mu_+^{R - r + 1}}{1 - \mu_+} + q \bigg(\frac{1 - q^r}{1 - q} \bigg) + \frac{q \mu_-}{1 - \mu_+} \bigg(\bigg(\frac{1 - q^r}{1 - q} \bigg) - \mu_+^{R - r + 1} \bigg(\frac{1 - (q \mu_+)^r}{1 - q \mu_+} \bigg) \bigg) \vspace*{8pt} \\ & \n = \n &
    \displaystyle \frac{1 - \mu_+^{R - r + 1}}{1 - \mu_+} + q \bigg(\frac{1 - q^r}{1 - q} \bigg) \bigg(\frac{1 - \mu_+ + \mu_-}{1 - \mu_+} \bigg) + \frac{q \mu_-}{1 - \mu_+} \bigg(- \mu_+^{R - r + 1} \bigg(\frac{1 - (q \mu_+)^r}{1 - q \mu_+} \bigg) \bigg) \vspace*{8pt} \\ & \n = \n &
    \displaystyle \frac{1}{1 - \mu_+} \bigg(1 + q \bigg(\frac{1 - q^r}{1 - q} \bigg)(1 - \mu_+ + \mu_-) - \mu_+^{R - r + 1} \bigg(1 + q \mu_- \bigg(\frac{1 - (q \mu_+)^r}{1 - q \mu_+} \bigg) \bigg) \bigg). \end{array} $$
 Using Lemma~\ref{lem:mean} also gives~$1 - \mu_+ + \mu_- = 1 - p$ and
 $$ 1 + q \mu_- \bigg(\frac{1 - (q \mu_+)^r}{1 - q \mu_+} \bigg) = \frac{1 - pq (1 + (\mu - 1)(\mu pq)^r)}{1 - \mu pq}. $$
 In conclusion,
 $$ E_r (S) = \frac{1}{1 - \mu p} \bigg(1 + q \bigg(\frac{1 - q^r}{1 - q} \bigg) (1 - p) - (\mu p)^{R - r + 1} \bigg(\frac{1 - pq (1 + (\mu - 1)(\mu pq)^r)}{1 - \mu pq} \bigg) \bigg), $$
 which completes the proof of Theorem~\ref{th:first}.

%%%%%%%%%%%%%%%%%%%%%%%%%%%%%%%%%%%%%%%%%%%%%%%%%%%%%%%%%%%%%%%%%%%%%%%%%%%%%%%%%%%%%%%%%%%%%%%%%%%%%%%%%%%%%%%%%%%%%%%%%%%%%%%%%%%%%%%%%%%%%%%%%%%%%%%%%%%%%%%%%%%%%%%%%%%%

\section{Proof of Theorem \ref{th:second} and~\eqref{eq:moment-2} (second moment)}
\label{sec:moment-2}
 To prove Theorem~\ref{th:second}, we follow the same strategy as for Theorem~\ref{th:first} using also that the numbers of infected vertices in disjoint subtrees are independent.
 Taking the square in~\eqref{eq:size},
 $$ \begin{array}{rcl} S^2 & \n = \n &
    \displaystyle S (x_r)^2 + 2 S (x_r) \ \sum_{i = 1}^k \,S (x_{r - i} \setminus x_{r - i + 1}) + \bigg(\sum_{i = 1}^k \,S (x_{r - i} \setminus x_{r - i + 1}) \bigg)^2 \vspace*{4pt} \\ & \n = \n &
    \displaystyle S (x_r)^2 + 2 S (x_r) \ \sum_{i = 1}^k \,S (x_{r - i} \setminus x_{r - i + 1}) \vspace*{0pt} \\ && \hspace*{25pt} + \
    \displaystyle \sum_{i = 1}^k \,S (x_{r - i} \setminus x_{r - i + 1})^2 + \sum_{i \neq j} \,S (x_{r - i} \setminus x_{r - i + 1}) S (x_{r - j} \setminus x_{r - j + 1}), \end{array} $$
 then conditioning on the event~$D = k$ and using independence of the random variables in~\eqref{eq:subtrees}~(because they represent the number of infected vertices in disjoint subtrees), we obtain
\begin{equation}
\label{eq:second}
  \begin{array}{rcl}
  \displaystyle E_r (S^2) = \sum_{k = 0}^r \bigg(E (S (x_r)^2) & \n + \n &
  \displaystyle 2 E (S (x_r)) \ \sum_{i = 1}^k \,E (S (x_{r - i} \setminus x_{r - i + 1})) \vspace*{-4pt} \\ & \n + \n &
  \displaystyle \sum_{i = 1}^k \,E (S (x_{r - i} \setminus x_{r - i + 1})^2) \vspace*{0pt} \\ & \n + \n &
  \displaystyle \sum_{i \neq j} \,E (S (x_{r - i} \setminus x_{r - i + 1})) \,E (S (x_{r - j} \setminus x_{r - j + 1})) \bigg) P (D = k). \end{array}
\end{equation}
 Then, using Lemmas~\ref{lem:plus} and~\ref{lem:minus}, and recalling from~\eqref{eq:D} that
 $$ P (D = k) = q_k \quad \hbox{for all} \quad k = 0, 1, \ldots, r, $$
 the right-hand side of~\eqref{eq:second} becomes
 $$ \begin{array}{rcl}
    \displaystyle E_r (S^2) = \sum_{k = 0}^r \bigg(\mu_{2, r} & \n + \n &
    \displaystyle 2 \mu_{1, r} \ \sum_{i = 1}^k \ (1 + \mu_- \mu_{1, r - i + 1}) \vspace*{-4pt} \\ & \n + \n &
    \displaystyle \sum_{i = 1}^k \ (1 + 2 \mu_- \mu_{1, r - i + 1} + \mu_- \mu_{2, r - i + 1} + (\sigma_-^2 + \mu_-^2 - \mu_-)(\mu_{1, r - i + 1})^2) \vspace*{0pt} \\ & \n + \n &
    \displaystyle \sum_{i \neq j} \,(1 + \mu_- \mu_{1, r - i + 1})(1 + \mu_- \mu_{1, r - j + 1}) \bigg) \,q_k \end{array} $$
 This completes the proof of the theorem. \\
\indent To simplify the previous expression for the second moment when the percolation process is subcritical~$\mu_+ = \mu p < 1$ and the local area network is infinite, and prove~\eqref{eq:moment-2}, we first observe that, by the monotone convergence theorem, the second moment on the infinite tree is equal to the limit of the second moment on the finite tree as the radius~$R \to \infty$.
 The reason why the expression simplifies in the infinite tree limit is because the terms~$\mu_{1, j}$ and~$\mu_{2, j}$ no longer depend on the index~$j$, which is due to the fact that they now represent the first and second moments of the number of infected vertices on infinite subtrees that are identically distributed.
 More precisely, taking the limit as~$R \to \infty$ in~\eqref{eq:parameters-3} and using Lemma~\ref{lem:plus}, we get
\begin{equation}
\label{eq:infinite-1}
  E (S (x_r)) = \mu_{1, r} = \frac{1}{1 - \mu_+} \quad \hbox{and} \quad E (S (x_r)^2) = \mu_{2, r} = \frac{1}{(1 - \mu_+)^2} \ \bigg(1 + \frac{\sigma_+^2}{1 - \mu_+} \bigg).
\end{equation}
 This, together with Lemma~\ref{lem:minus}, implies that
\begin{equation}
\label{eq:infinite-2}
  \begin{array}{rcl}
    E (S (x_{r - i} \setminus x_{r - i + 1})) & \n = \n & 1 + \mu_- \mu_{1, r - i + 1} =
  \displaystyle 1 + \frac{\mu_-}{1 - \mu_+} \vspace*{4pt} \\
    E (S (x_{r - i} \setminus x_{r - i + 1})^2) & \n = \n & 1 + 2 \mu_- \mu_{1, r - i + 1} + \mu_- \mu_{2, r - i + 1} + (\sigma_-^2 + \mu_-^2 - \mu_-)(\mu_{1, r - i + 1})^2 \vspace*{8pt} \\ & \n = \n &
  \displaystyle 1 + \frac{2 \mu_-}{1 - \mu_+} + \frac{\mu_-}{(1 - \mu_+)^2} \ \bigg(1 + \frac{\sigma_+^2}{1 - \mu_+} \bigg) + \frac{\sigma_-^2 + \mu_-^2 - \mu_-}{(1 - \mu_+)^2} \vspace*{8pt} \\ & \n = \n &
  \displaystyle 1 + \frac{2 \mu_-}{1 - \mu_+} + \frac{\sigma_-^2 + \mu_-^2}{(1 - \mu_+)^2} + \frac{\sigma_+^2 \mu_-}{(1 - \mu_+)^3}. \end{array}
\end{equation}
 Using that, in the limit as~$R \to \infty$, the terms in the two sums over~$i$ and the terms in the sum over~$i \neq j$ in equation~\eqref{eq:second} are constant given by~\eqref{eq:infinite-1} and~\eqref{eq:infinite-2}, we obtain
 $$ \begin{array}{rcl}
    \displaystyle E_r (S^2) & \n = \n &
    \displaystyle \frac{1}{(1 - \mu_+)^2} \ \bigg(1 + \frac{\sigma_+^2}{1 - \mu_+} \bigg) + \frac{2}{1 - \mu_+} \bigg(1 + \frac{\mu_-}{1 - \mu_+} \bigg) E (D) \vspace*{8pt} \\ && \hspace*{5pt} + \
    \displaystyle \bigg(1 + \frac{2 \mu_-}{1 - \mu_+} + \frac{\sigma_-^2 + \mu_-^2}{(1 - \mu_+)^2} + \frac{\sigma_+^2 \mu_-}{(1 - \mu_+)^3} \bigg) E (D) + \bigg(1 + \frac{\mu_-}{1 - \mu_+} \bigg)^2 E (D (D - 1)) \vspace*{8pt} \\ & \n = \n &
    \displaystyle \frac{1}{(1 - \mu_+)^2} \ \bigg(1 + \frac{\sigma_+^2}{1 - \mu_+} \bigg) +
    \displaystyle \bigg(1 + \frac{2 (1 + \mu_-)}{1 - \mu_+} + \frac{2 \mu_- + \sigma_-^2 + \mu_-^2}{(1 - \mu_+)^2} + \frac{\sigma_+^2 \mu_-}{(1 - \mu_+)^3} \bigg) E (D) \vspace*{8pt} \\ && \hspace*{5pt} + \
    \displaystyle \bigg(1 + \frac{\mu_-}{1 - \mu_+} \bigg)^2 E (D (D - 1)). \end{array} $$
 Then, using Lemma~\ref{lem:D}, we get
 $$ \begin{array}{rcl}
    \displaystyle E_r (S^2) & \n = \n &
    \displaystyle \frac{1}{(1 - \mu_+)^2} \ \bigg(1 + \frac{\sigma_+^2}{1 - \mu_+} \bigg) \vspace*{8pt} \\ && \hspace*{5pt} + \
    \displaystyle \bigg(1 + \frac{2 (1 + \mu_-)}{1 - \mu_+} + \frac{2 \mu_- + \sigma_-^2 + \mu_-^2}{(1 - \mu_+)^2} + \frac{\sigma_+^2 \mu_-}{(1 - \mu_+)^3} \bigg) \ q \bigg(\frac{1 - q^r}{1 - q} \bigg) \vspace*{8pt} \\ && \hspace*{5pt} + \
    \displaystyle \bigg(1 + \frac{\mu_-}{1 - \mu_+} \bigg)^2 \ \frac{2q^2 (1 - rq^{r - 1} + (r - 1) q^r)}{(1 - q)^2}. \end{array} $$
 and finally Lemmas~\ref{lem:mean} and~\ref{lem:variance},
 $$ \begin{array}{rcl}
    \displaystyle E_r (S^2) & \n = \n &
    \displaystyle \frac{1}{(1 - \mu p)^2} \ \bigg(1 + \frac{p (1 - p) \mu + p^2 \sigma^2}{1 - \mu p} \bigg) \vspace*{8pt} \\ && \hspace*{5pt} + \
    \displaystyle \bigg(1 + \frac{2 (1 + (\mu - 1) p)}{1 - \mu p} + \frac{2 (\mu - 1) p + p (1 - p)(\mu - 1) + p^2 \sigma^2 + (\mu - 1)^2 p^2}{(1 - \mu p)^2} \vspace*{8pt} \\ && \hspace*{100pt} + \
    \displaystyle \frac{(p (1 - p) \mu + p^2 \sigma^2)(\mu - 1) p}{(1 - \mu p)^3} \bigg) \ q \bigg(\frac{1 - q^r}{1 - q} \bigg) \vspace*{8pt} \\ && \hspace*{5pt} + \
    \displaystyle \bigg(1 + \frac{(\mu - 1) p}{1 - \mu p} \bigg)^2 \ \frac{2q^2 (1 - rq^{r - 1} + (r - 1) q^r)}{(1 - q)^2}, \end{array} $$
 which completes the proof of~\eqref{eq:moment-2}.

%%%%%%%%%%%%%%%%%%%%%%%%%%%%%%%%%%%%%%%%%%%%%%%%%%%%%%%%%%%%%%%%%%%%%%%%%%%%%%%%%%%%%%%%%%%%%%%%%%%%%%%%%%%%%%%%%%%%%%%%%%%%%%%%%%%%%%%%%%%%%%%%%%%%%%%%%%%%%%%%%%%%%%%%%%%%

\section{Proof of Theorem \ref{th:decay} (exponential decay)}
 To prove exponential decay of the diameter of the cluster of infected vertices, the idea is to study the process that keeps track of the random number of infected vertices at distance~$n$ from the highest infected vertex in the tree.
 More precisely, on the event~$D = k$, we let
 $$ X_n = \card (\C_n) \quad \hbox{where} \quad \C_n = \{x \in \C : d (x, x_{r - k}) = n \}. $$
 Recall that~$x_{r - k}$ is the unique vertex along the path connecting the source of the infection and the root of the tree that is at distance~$k$ from the source of the infection.
 The next lemma shows that, in the subcritical phase~$\mu p < 1$, the expected value of the process decays exponentially.
\begin{lemma}-- \
\label{lem:decay}
 Given that the infection starts at distance~$r$ from the root,
 $$ E_r (X_{n + 1}) = \left\{\begin{array}{ccl}
                      \mu p \,E (X_n) + (1 - p) \,q^{n + 1} & \hbox{for} & n < r \vspace*{2pt} \\
                      \mu p \,E (X_n)                       & \hbox{for} & n \geq r. \end{array} \right. $$
\end{lemma}
\begin{proof}
 Recall that each vertex produces~$\mu$ offspring in average and that each of the offspring of an infected vertex is infected with probability~$p$, from which it follows that each infected vertex has~$\mu p$ infected offspring in average.
 Now, given~$D = k$, the process is conditioned so that
 $$ x_{r - k} \in \C_0, \quad x_{r - k + 1} \in \C_1, \quad \ldots \quad x_{r - 1} \in \C_{k - 1} \quad \hbox{and} \quad  x_r \in \C_k $$
 are infected so, until generation~$k - 1$, the vertices in~$\C_n$ have~$\mu p$ infected offspring in average except for vertex~$x_{r - k + n}$ that has~$(\mu - 1) p + 1$ infected offspring in average.
 This implies that
 $$ E_r (X_{n + 1} \,| \,X_n, D = k) = \left\{\begin{array}{ccl}
                                       \mu p (X_n - 1) + (\mu - 1) p + 1 & \hbox{for} & n < k \vspace*{2pt} \\
                                       \mu p X_n     & \hbox{for} & n \geq k. \end{array} \right. $$
 Conditioning on the random variable~$D$, we deduce that
 $$ \begin{array}{rcl}
     E_r (X_{n + 1} \,| \,X_n) & \n = \n &
    \displaystyle \sum_{k = 0}^{\infty} \,E_r (X_{n + 1} \,| \,X_n, D = k) P_r (D = k) \vspace*{4pt} \\ & \n = \n &
    \displaystyle \sum_{k = 0}^n \ (\mu p X_n) \,P_r (D = k) + \sum_{k = n + 1}^{\infty} (\mu p (X_n - 1) + (\mu - 1) p + 1) \,P_r (D = k) \vspace*{4pt} \\ & \n = \n &
    \displaystyle \sum_{k = 0}^n \ (\mu p X_n) \,P_r (D = k) + \sum_{k = n + 1}^{\infty} (\mu p X_n + 1 - p) \,P_r (D = k) \vspace*{4pt} \\ & \n = \n &
    \displaystyle \sum_{k = 0}^{\infty} \ (\mu p X_n) \,P_r (D = k) + (1 - p) \sum_{k = n + 1}^{\infty} P_r (D = k) =
    \mu p X_n + (1 - p) P_r (D > n). \end{array} $$
 Recalling the probability mass function of~$D$, we deduce that
 $$ E_r (X_{n + 1}) = E (E (X_{n + 1} \,| \,X_n)) = \left\{\begin{array}{ccl}
                                                    \mu p \,E (X_n) + (1 - p) \,q^{n + 1} & \hbox{for} & n < r \vspace*{2pt} \\
                                                    \mu p \,E (X_n)                       & \hbox{for} & n \geq r. \end{array} \right. $$
 This completes the proof.
\end{proof} \\ \\
 It follows from the lemma that, for all~$n \leq r$,
 $$ \begin{array}{rcl}
     E_r (X_n) & \n = \n & (\mu p) \,E_r (X_{n - 1}) + (1 - p) \,q^n \vspace*{4pt} \\
               & \n = \n & (\mu p)^2 E_r (X_{n - 2}) + (\mu p) (1 - p) \,q^{n - 1} + (1 - p) \,q^n \vspace*{4pt} \\
               & \n = \n & (\mu p)^3 E_r (X_{n - 3}) + (\mu p)^2 (1 - p) \,q^{n - 2} + (\mu p)(1 - p) \,q^{n - 1} + (1 - p) \,q^n \vspace*{4pt} \\
               & \n = \n & (\mu p)^n E_r (X_0) + (1 - p)((\mu p)^{n - 1} \,q + (\mu p)^{n - 2} \,q^2 + \cdots + (\mu p) \,q^{n - 1} + q^n) \vspace*{4pt} \\
               & \n \leq \n & (\mu p)^n E_r (X_0) + (\mu p)^{n - 1} \,q + (\mu p)^{n - 2} \,q^2 + \cdots + (\mu p) \,q^{n - 1} + q^n. \end{array} $$
 Then, using that~$E_r (X_0) = 1$, we get
 $$ E_r (X_n) \leq \sum_{k = 0}^n \,(\mu p)^{n - k} q^k = (\mu p)^n \,\sum_{k = 0}^n \bigg(\frac{q}{\mu p} \bigg)^k = \frac{1 - (q / \mu p)^{n + 1}}{1 - (q / \mu p)} \ (\mu p)^n \quad \hbox{for all} \quad n \leq r. $$
 Observing also that, for all~$n > r$,
 $$ E_r (X_n) \leq (\mu p) \,E_r (X_{n - 1}) \leq (\mu p)^2 E_r (X_{n - 2}) \leq \cdots \leq (\mu p)^{n - r} E_r (X_r) $$
 we conclude that, for all~$n > r$, the diameter exceeds~$2n$ with probability
 $$ \begin{array}{rcl} P_r (\diam (\C) \geq 2n) & \n \leq \n &
    \displaystyle P_r (X_n > 0) = \sum_{k = 1}^{\infty} \,P_r (X_n = k) \leq \sum_{k = 1}^{\infty} \,k \,P_r (X_n = k) = E_r (X_n) \vspace*{8pt} \\ & \n \leq \n &
    \displaystyle (\mu p)^{n - r} E_r (X_r) \leq (\mu p)^{n - r} \ \frac{1 - (q / \mu p)^{r + 1}}{1 - (q / \mu p)} \ (\mu p)^r =
    \displaystyle \frac{1 - (q / \mu p)^{r + 1}}{1 - (q / \mu p)} \ (\mu p)^n. \end{array} $$
 This completes the proof of Theorem~\ref{th:decay}.

%%%%%%%%%%%%%%%%%%%%%%%%%%%%%%%%%%%%%%%%%%%%%%%%%%%%%%%%%%%%%%%%%%%%%%%%%%%%%%%%%%%%%%%%%%%%%%%%%%%%%%%%%%%%%%%%%%%%%%%%%%%%%%%%%%%%%%%%%%%%%%%%%%%%%%%%%%%%%%%%%%%%%%%%%%%%

%\section{Conclusion}
%\hspace{\parindent} In this paper, as a main theoretical contribution to the existing financial market modeling literature, we develop a structural percolation model for the spread of infection within a financial market that is modeled as a tree structure. From our results, we found the exact first moment of the cluster size on a finite tree and the second moment on the infinite tree. Therefore, these moments can be used to find the mean and variance of number of companies that would possess the new infection in order to assess the price of the shares or any other financial securities. In addition, we found that the tail distribution of the radius of the infection decays exponentially.

%There are many opportunities for further research following this work. One is to include loops in the tree, so that traders can buy from multiple sources. Another is to consider another type of structure instead of a tree e.g. a star with bidirectional probabilities.

%%%%%%%%%%%%%%%%%%%%%%%%%%%%%%%%%%%%%%%%%%%%%%%%%%%%%%%%%%%%%%%%%%%%%%%%%%%%%%%%%%%%%%%%%%%%%%%%%%%%%%%%%%%%%%%%%%%%%%%%%%%%%%%%%%%%%%%%%%%%%%%%%%%%%%%%%%%%%%%%%%%%%%%%%%%%

\bibliographystyle{plain}
\bibliography{biblio.bib}

\begin{thebibliography}{10}

\bibitem{1254}
{U.K. cyber insurance trends 2020}, June 2021.

\bibitem{1246}
National~Security Agency.
\newblock Defend privileges and accounts, August 2019.

\bibitem{1256}
Zeinab Amin.
\newblock A practical road map for assessing cyber risk.
\newblock {\em Journal of Risk Research}, 22(1):32--43, 2019.

\bibitem{1259}
Y~Antonio and SW~Indratno.
\newblock Cyber insurance rate making based on markov model for regular
  networks topology.
\newblock In {\em Journal of Physics: Conference Series}, volume 1752, page
  012002. IOP Publishing, 2021.

\bibitem{1248}
{Australian Cyber Security Centre}.
\newblock Restricting administrative privileges, October 2021.

\bibitem{betterley}
R.S. Betterley.
\newblock Cyber / privacy insurance market survey: A tough market for larger
  insureds, but smaller insureds finding eager insurers, June 2016.

\bibitem{eling_wirfs}
Martin Eling and Jan~Hendrik Wirfs.
\newblock Modelling and management of cyber risk.
\newblock {\em International Actuarial Association Life Section}, 2015.

\bibitem{1247}
Hildegard Ferraiolo, David~A Cooper, Andrew~R Regenscheid, Karen Scarfone, and
  Murugiah~P Souppaya.
\newblock Best practices for privileged user piv authentication, April 2016.

\bibitem{herath_herath}
Hemantha Herath and Tejaswini Herath.
\newblock Copula-based actuarial model for pricing cyber-insurance policies.
\newblock {\em Insurance markets and companies: analyses and actuarial
  computations}, 2(1):7--20, 2011.

\bibitem{jevtic_lanchier_2020}
P.~Jevti\'{c} and N.~Lanchier.
\newblock Dynamic structural percolation model of loss distribution for cyber
  risk of small and medium-sized enterprises for tree-based {LAN} topology.
\newblock {\em Insurance Math. Econom.}, 91:209--223, 2020.

\bibitem{1250}
Terrence~J Moore and Jin-Hee Cho.
\newblock Applying percolation theory.
\newblock In {\em Cyber Resilience of Systems and Networks}, pages 107--133.
  Springer, 2019.

\bibitem{intrusion_nist}
{National Institute of Standards and Technology}.
\newblock Intrusion, 2021.

\bibitem{1255}
NetDiligence.
\newblock Cyber claims study, 2019.

\bibitem{jorion}
Jorion Philippe.
\newblock Value at risk: the new benchmark for managing financial risk.
\newblock {\em NY: McGraw-Hill Professional}, 2006.

\bibitem{1249}
Karen Scarfone and Paul Hoffman.
\newblock Guidelines on firewalls and firewall policy, September 2009.

\bibitem{segal}
Sim Segal.
\newblock {\em Corporate value of Enterprise risk management: the next step in
  business management}, volume~3.
\newblock John Wiley \& Sons, 2011.

\bibitem{1252}
SonicWall.
\newblock Mid-year update: Sonicwall cyber threat report, June 2021.

\bibitem{The_Institute_of_Risk_Management_2018}
{The Institute of Risk Management}.
\newblock Cyber risk and risk management, 2018.

\bibitem{1258}
{U.S. Government Accountability Office}.
\newblock Cyber insurance: insurers and policyholders face challenges in an
  evolving market, May 2021.

\bibitem{1260}
{Verizon}.
\newblock 2018 verizon data breach investigations report, April 2018.

\bibitem{1253}
{Verizon}.
\newblock 2021 verizon data breach investigations report, June 2021.

\bibitem{1257}
Maochao Xu and Lei Hua.
\newblock Cybersecurity insurance: Modeling and pricing.
\newblock {\em North American Actuarial Journal}, 23(2):220--249, 2019.

\end{thebibliography}

\end{document}